\documentclass[11pt]{amsart}

\usepackage[T1]{fontenc}
\usepackage{lmodern}
\usepackage{amsmath,amssymb,amsthm,mathtools}
\usepackage{graphicx}
\usepackage{microtype}
\usepackage[hidelinks]{hyperref}

\newtheorem{theorem}{Theorem}[section]
\newtheorem{lemma}[theorem]{Lemma}
\newtheorem{corollary}[theorem]{Corollary}
\newtheorem{proposition}[theorem]{Proposition}
\theoremstyle{remark}
\newtheorem{remark}[theorem]{Remark}

\newcommand{\D}{\Delta}
\newcommand{\Int}{\operatorname{int}}
\newcommand{\Nbd}{\mathcal N}
\newcommand{\cl}{\operatorname{cl}}

\title[Reducible handle additions]{Reducible handle additions along separating slopes}
\author{Qilong Guo}
\thanks{College of Science, China University of Petroleum-Beijing, Beijing 102249, P. R. China. Corresponding author: guoqilong1984@hotmail.com.}
\address{College of Science, China University of Petroleum-Beijing, Beijing 102249, P. R. China}
\email{guoqilong1984@hotmail.com}
\date{}


\begin{document}

\begin{abstract}
Let $M$ be a compact orientable simple $3$--manifold, and let $F$ be a boundary component of genus at least two. Qiu and Zhang proved that if $\alpha,\beta\subset F$ are separating slopes and both $M(\alpha)$ and $M(\beta)$ are reducible, then $\D(\alpha,\beta)\le2$. We prove that in fact $\D(\alpha,\beta)=0$, settling the reducible--reducible case of their conjecture on separating handle additions.
\end{abstract}

\maketitle

\section{Introduction}

Let $M$ be a compact orientable simple $3$--manifold, and let $F$ be a component of $\partial M$ of genus at least two. Here \emph{simple} means irreducible, boundary-irreducible, atoroidal, and anannular. By a \emph{slope} on $F$ we mean an essential unoriented simple closed curve in $F$. For a slope $\gamma$, write $M(\gamma)$ for the manifold obtained by attaching a $2$--handle along an annular neighborhood of $\gamma$. For slopes $\alpha$ and $\beta$, let $\D(\alpha,\beta)$ denote the minimum of $|\alpha'\cap\beta'|$ over curves $\alpha'$ and $\beta'$ isotopic in $F$ to $\alpha$ and $\beta$, respectively.

Scharlemann and Wu \cite{SW} studied slopes for which handle addition destroys simplicity, calling such slopes \emph{degenerating}. They conjectured that two degenerating slopes are either disjoint or can be isotoped into a common once-punctured torus \cite[Conjecture~1]{SW}. When both slopes are separating, the once-punctured-torus alternative cannot occur. They also proved that only finitely many isotopy classes of separating degenerating slopes occur on a fixed higher-genus boundary component \cite[Theorem~3.4]{SW}, and that in the mixed reducible--$\partial$-reducible case one has $\D(\alpha,\beta)=0$ whenever one of the slopes is separating \cite[Corollary~4.4]{SW}.

Building on this work, Qiu and Zhang conjectured that $\D(\alpha,\beta)=0$ whenever $\alpha$ and $\beta$ are separating and each of $M(\alpha)$ and $M(\beta)$ is reducible or $\partial$-reducible \cite[Conjecture~1]{QZ}. In the reducible--reducible case, Zhang, Qiu, and Li obtained the bound $\D(\alpha,\beta)\le4$ \cite{ZQL}, and Qiu and Zhang subsequently improved it to $\D(\alpha,\beta)\le2$ \cite{QZ}. Our main theorem sharpens this bound to $\D(\alpha,\beta)=0$, thereby settling the reducible--reducible case of the Qiu--Zhang conjecture.

\begin{theorem}\label{thm:main}
Let $M$ be a compact orientable simple $3$--manifold, and let $F$ be a component of $\partial M$ with $g(F)\ge2$. Suppose that $\alpha,\beta\subset F$ are separating slopes and that both $M(\alpha)$ and $M(\beta)$ are reducible. Then $\D(\alpha,\beta)=0$.
\end{theorem}

Together with the standard bound for pairwise disjoint separating curves on $F$, Theorem~\ref{thm:main} gives the following consequence.

\begin{corollary}\label{cor:number}
Under the hypotheses on $M$ and $F$ in Theorem~\ref{thm:main}, there are at most $2g(F)-3$ isotopy classes of separating slopes $\gamma\subset F$ for which $M(\gamma)$ is reducible.
\end{corollary}

Beyond the reducible case, several other exceptional outcomes of handle addition have been studied. $\partial$-reducible additions were treated by Li, Qiu, and Zhang \cite{LQZ} and by Lou and Zhang \cite{LZ}. Eudave-Mu\~noz \cite{EM} considered annular and toroidal outcomes, while Chan-Palomo \cite{CP22,CP23} studied toroidal additions and mixed toroidal--reducible pairs.

The paper is organized as follows. Section~2 collects the preliminary results on reducible handle additions, sutured manifolds, and Taylor's handle-addition theorem. Since separating slopes have even geometric intersection number, the Qiu--Zhang bound reduces the proof to excluding the case $\D(\alpha,\beta)=2$. Section~3 shows that a counterexample with $g(F)$ minimal must have $g(F)=4$ and may be chosen so that $F$ is the only non-toral boundary component. Section~4 constructs a taut sutured structure in this genus-four configuration and applies Taylor's theorem to obtain a contradiction.

\section{Preliminaries}

Throughout the paper, all manifolds and surfaces are compact and orientable, and all $3$--manifolds are connected, unless otherwise stated.

\subsection{Reducible handle additions and elementary facts}

The main input for the genus reduction is the distance bound of Qiu and Zhang \cite[Theorem~1]{QZ}.

\begin{theorem}[Qiu--Zhang]\label{thm:QZ}
Let $M$ be a simple $3$--manifold, and let $F$ be a boundary component of $\partial M$ of genus at least two. If $\alpha,\beta\subset F$ are separating slopes and both $M(\alpha)$ and $M(\beta)$ are reducible, then $\D(\alpha,\beta)\le2$.
\end{theorem}

The following results of Scharlemann and Wu \cite[Theorem~3.4 and Corollary~4.4]{SW} will also be used.

\begin{lemma}[Scharlemann--Wu]\label{lem:SW}
Let $M$ be a simple $3$--manifold, and let $F$ be a boundary component of $\partial M$ of genus at least two. Then the following hold.
\begin{enumerate}
\item There are only finitely many isotopy classes of separating degenerating slopes on $F$.
\item Suppose that $\alpha,\beta\subset F$ are slopes, $M(\alpha)$ is reducible, and $M(\beta)$ is $\partial$-reducible. If either $\alpha$ or $\beta$ is separating, then $\D(\alpha,\beta)=0$.
\end{enumerate}
\end{lemma}

This statement immediately implies the following.

\begin{corollary}\label{cor:filled-boundary-irred}
Let $M$ and $F$ be as in Lemma~\ref{lem:SW}, and let $\alpha,\beta\subset F$ be separating slopes such that both $M(\alpha)$ and $M(\beta)$ are reducible. If $\D(\alpha,\beta)>0$, then both $M(\alpha)$ and $M(\beta)$ are boundary-irreducible.
\end{corollary}

The following elementary lemma shows that reducibility is preserved under an additional $2$--handle attachment along a boundary component of genus at least two.

\begin{lemma}\label{lem:reducibility-preserved}
Let $N$ be a reducible $3$--manifold, let $G\subset\partial N$ be a boundary component of genus at least two, and let $\delta$ be a slope on $G$. Then $N(\delta)$ is reducible.
\end{lemma}

\begin{proof}
Let $S\subset\Int N$ be a reducing sphere. The attached $2$--handle is disjoint from $S$. If $S$ is nonseparating, then it remains nonseparating in $N(\delta)$ and hence remains a reducing sphere. If $S$ is separating, let $X$ and $Y$ be the closures of the components of $N-S$, with $G\subset\partial X$. The component $Y$ is unchanged and is not a $3$--ball. Since $g(G)\ge2$ and $\delta$ is essential, compressing $G$ along $\delta$ leaves a positive-genus boundary component of $X(\delta)$. Thus $X(\delta)$ is not a $3$--ball either, and $S$ remains a reducing sphere in $N(\delta)$.
\end{proof}

\subsection{Sutured manifolds and Taylor's handle-addition theorem}

We adopt Gabai's sutured-manifold terminology \cite{Gabai}, using the notation of Taylor \cite[Section~3]{Taylor}. A \emph{sutured manifold} is a pair $(N,\gamma)$ satisfying the following conditions. The manifold $N$ has nonempty boundary, and $\gamma$ is a finite collection of pairwise disjoint oriented simple closed curves in the non-toral components of $\partial N$, called the \emph{sutures}. Let $A(\gamma)\subset\partial N$ be a union of pairwise disjoint annular neighborhoods of the sutures, let $T(\gamma)$ be the union of the torus components of $\partial N$, and set
\[
R(\gamma)=\cl\bigl(\partial N-(A(\gamma)\cup T(\gamma))\bigr).
\]
The surface $R(\gamma)$ is required to admit a decomposition
\[
R(\gamma)=R_-(\gamma)\sqcup R_+(\gamma)
\]
so that every component of $A(\gamma)$ is adjacent to one component of $R_-(\gamma)$ and one component of $R_+(\gamma)$. We orient $R_+(\gamma)$ outward and $R_-(\gamma)$ inward; the induced boundary orientations agree with the orientations of the corresponding sutures.

For an oriented surface $S=\bigsqcup_jS_j$, with each $S_j$ connected, set
\[
x(S)=\sum_j\max\{0,-\chi(S_j)\}.
\]
The sutured manifold $(N,\gamma)$ is \emph{taut} if $N$ is irreducible, every component of $T(\gamma)$ is incompressible, and each of $R_+(\gamma)$ and $R_-(\gamma)$ is incompressible and minimizes $x$ among oriented surfaces with the same oriented boundary and representing the same relative homology class.

A properly embedded surface $Q\subset N$ is a \emph{parameterizing surface} if
\begin{enumerate}
\item[(P1)] $\partial Q\cap A(\gamma)$ consists of spanning arcs, each intersecting the corresponding core suture once;
\item[(P2)] no component of $Q$ is a sphere or a disk disjoint from $\gamma$.
\end{enumerate}

Let $b$ be a component of $\gamma$. Write $N[b]$ for the manifold obtained by attaching a $2$--handle along $b$, and let $\kappa_b\subset N[b]$ denote the cocore arc of this handle. Following Taylor \cite[Introduction]{Taylor}, a \emph{$b$--boundary-compressing disk} for $Q$ is an embedded disk $D\subset N$ with $\partial D=a\cup d$, where $a\subset Q$ is an arc and $d\subset b$ is a subarc, and with $\Int D\cap(Q\cup\partial N)=\varnothing$. The arc $a$ is not required to be essential in $Q$. The following special case of Taylor's handle-addition theorem \cite[Theorem~9.5]{Taylor} will be used.

\begin{theorem}[Taylor]\label{thm:Taylor}
Let $(N,\gamma)$ be a taut sutured manifold, and let $b$ be a component of $\gamma$ whose two adjacent components of $R(\gamma)$ are thrice-punctured spheres. Let $Q\subset N$ be a parameterizing surface with $|Q\cap b|\ge1$, and suppose that the intersection of $Q$ with each of the two components of $R(\gamma)$ adjacent to $b$ contains no inessential arcs or circles. Let $\kappa_b\subset N[b]$ be the cocore arc of the attached $2$--handle. Then at least one of the following holds:
\begin{enumerate}
\item $Q$ has a compressing disk or a $b$--boundary-compressing disk;
\item $(N[b],\kappa_b)=(M'_0,\kappa'_0)\#(M'_1,\kappa'_1)$, where $M'_1$ is a lens space and $\kappa'_1$ is a core of a genus-one Heegaard splitting of $M'_1$;
\item $(N[b],\gamma-b)$ is taut;
\item $-2\chi(Q)+|Q\cap\gamma|\ge2|Q\cap b|$.
\end{enumerate}
\end{theorem}

\begin{remark}\label{rmk:Taylor-summand}
In alternative~(2), the symbol $\#$ denotes the connected sum of pairs. Thus there is a separating sphere $S\subset N[b]$, transverse to $\kappa_b$, with $|S\cap\kappa_b|=2$. It is important to note that the lens-space/core pair is a nontrivial summand. In particular, the corresponding punctured summand is not a $3$--ball containing a boundary-parallel arc of $\kappa_b$; see \cite[Lemma~8.3 and its proof]{Taylor}.
\end{remark}

\section{Reduction to genus four}

Suppose, for contradiction, that Theorem~\ref{thm:main} is false, and set
\[
\begin{aligned}
\mathcal C=\bigl\{(M,F,\alpha,\beta)\ \bigm|\ {}&
M\text{ is simple},\\
&F\text{ is a component of }\partial M,\quad g(F)\ge2,\\
&\alpha,\beta\subset F\text{ are separating slopes},\\
&M(\alpha)\text{ and }M(\beta)\text{ are reducible},\quad
\D(\alpha,\beta)>0
\bigr\}.
\end{aligned}
\]
Choose a counterexample $(M,F,\alpha,\beta)\in\mathcal C$ for which $g(F)$ is minimal.

\begin{proposition}\label{prop:genus4}
For the chosen quadruple, $\D(\alpha,\beta)=2$ and $g(F)=4$. Moreover, after isotopy,
\[
F-\Int\Nbd_F(\alpha\cup\beta)=X_1\sqcup X_2\sqcup X_3\sqcup X_4,
\]
where each $X_i$ is a once-punctured torus.
\end{proposition}

\begin{proof}
Since $\alpha$ is separating, $\D(\alpha,\beta)$ is even. As $\D(\alpha,\beta)>0$, Theorem~\ref{thm:QZ} therefore gives $\D(\alpha,\beta)=2$. Without loss of generality, assume that $|\alpha\cap\beta|=2$, and set $P=\Nbd_F(\alpha\cup\beta)$. Then $P$ is a four-punctured sphere and
\[
F-\Int P=X_1\sqcup X_2\sqcup X_3\sqcup X_4,
\]
where each $X_i$ has one boundary component and genus $h_i\ge1$. Label the components so that $\alpha$ separates $X_1\cup X_2$ from $X_3\cup X_4$, while $\beta$ separates $X_1\cup X_3$ from $X_2\cup X_4$. In particular,
\[
g(F)=h_1+h_2+h_3+h_4.
\]

It remains to show that $h_i=1$ for every $i$. Suppose otherwise. Without loss of generality, assume that $h_1\ge2$. There are infinitely many separating curves $\delta\subset\Int X_1$, pairwise nonisotopic in $F$, each of which bounds a once-punctured torus in $X_1$. By Lemma~\ref{lem:SW}(1), we may choose such a $\delta$ so that $M'=M(\delta)$ is simple. Compressing $F$ along $\delta$ yields a torus and a boundary component $F'$ of genus $g(F)-1$.

\medskip
\noindent\emph{Claim 1.} $(M',F',\alpha,\beta)\in\mathcal C$.

\smallskip
\noindent\emph{Proof of Claim 1.}
The component $F'$ is obtained from $F$ by replacing $X_1$ with a surface of genus $h_1-1\ge1$ having one boundary component; $P$, $\alpha$, $\beta$, and $X_2,X_3,X_4$ are unchanged. Consequently $F'$ contains $\alpha$ and $\beta$ as essential separating slopes. Any new bigon would give a disk component of $F'-(\alpha\cup\beta)$, whereas all four complementary components of $P$ have positive genus. Thus $\D_{F'}(\alpha,\beta)=2$.

It remains to verify that $M'(\alpha)$ and $M'(\beta)$ are reducible. Since $\delta$ is disjoint from $\alpha$, the handle additions commute:
\[
M'(\alpha)\cong M(\alpha)(\delta).
\]
The $\alpha$--handle attachment caps the boundary of the component of $F-\alpha$ containing $X_1\cup X_2$, producing a boundary component of $M(\alpha)$ of genus $h_1+h_2\ge3$. The curve $\delta$ remains essential on this component. Lemma~\ref{lem:reducibility-preserved} therefore shows that $M(\alpha)(\delta)$, and hence $M'(\alpha)$, is reducible. The same argument with $\beta$ in place of $\alpha$ shows that $M'(\beta)$ is reducible. Thus $(M',F',\alpha,\beta)\in\mathcal C$. \qed

Since $g(F')=g(F)-1$, Claim~1 contradicts the minimality of $g(F)$. Hence $h_i=1$ for every $i$, and therefore $g(F)=4$.
\end{proof}

\begin{lemma}\label{lem:boundary-normalized}
There exists a quadruple $(M,F,\alpha,\beta)\in\mathcal C$ such that $g(F)=4$ and $F$ is the only non-toral component of $\partial M$.
\end{lemma}

\begin{proof}
Let $(M,F,\alpha,\beta)\in\mathcal C$ be the minimal counterexample considered in Proposition~\ref{prop:genus4}. Since $M$ is irreducible and has a boundary component of positive genus, $\partial M$ has no sphere components. Suppose that $G\ne F$ is a boundary component with $h=g(G)\ge2$. There are infinitely many separating slopes on $G$, pairwise nonisotopic in $G$, each bounding a once-punctured torus. By Lemma~\ref{lem:SW}(1), choose one, say $\delta$, such that $M'=M(\delta)$ is simple.

Since $M'$ is simple by construction and the data $F,\alpha,\beta$ are unchanged, it remains only to verify reducibility of $M'(\alpha)$ and $M'(\beta)$. Since $\delta\subset G$ and $\alpha,\beta\subset F$, the handle additions commute:
\[
M'(\alpha)=M(\delta)(\alpha)\cong M(\alpha)(\delta),\qquad
M'(\beta)=M(\delta)(\beta)\cong M(\beta)(\delta).
\]
Because the handle additions along $\alpha$ and $\beta$ are supported on $F$, the boundary component $G$ is unchanged in both filled manifolds. Lemma~\ref{lem:reducibility-preserved} therefore shows that $M(\alpha)(\delta)$ and $M(\beta)(\delta)$ are reducible. It follows that $M'(\alpha)$ and $M'(\beta)$ are reducible, and hence $(M',F,\alpha,\beta)\in\mathcal C$.

Repeating the construction finitely many times yields a quadruple in $\mathcal C$ for which every boundary component other than $F$ is a torus.
\end{proof}

\section{The genus-four contradiction}

Fix a quadruple $(M,F,\alpha,\beta)\in\mathcal C$ as in Lemma~\ref{lem:boundary-normalized}, and retain the notation and labeling of Proposition~\ref{prop:genus4}. By Corollary~\ref{cor:filled-boundary-irred}, both $M(\alpha)$ and $M(\beta)$ are boundary-irreducible. Set $c_i=\partial X_i$ for $1\le i\le4$.

\subsection{A taut sutured structure}

Cutting $P$ along $\alpha$ gives two pairs of pants $P_{12}$ and $P_{34}$ with
\[
\partial P_{12}=\alpha\cup c_1\cup c_2,\qquad
\partial P_{34}=\alpha\cup c_3\cup c_4.
\]
Set $\gamma=\{\alpha,c_1,c_2,c_3,c_4\}$, and let $A(\gamma)$ be a union of pairwise disjoint annular neighborhoods of these curves. Let $T(\gamma)$ be the union of the torus components of $\partial M-F$. Orient the sutures so that the complementary regions have the signs shown in Figure~\ref{fig:sutures}; thus
\[
R_+(\gamma)=\cl\bigl((P_{12}\sqcup X_3\sqcup X_4)-A(\gamma)\bigr),
\qquad
R_-(\gamma)=\cl\bigl((P_{34}\sqcup X_1\sqcup X_2)-A(\gamma)\bigr).
\]
Because every annulus of $A(\gamma)$ is adjacent to one positive and one negative region, these choices define a sutured manifold. Each of $R_+(\gamma)$ and $R_-(\gamma)$ is homeomorphic to the disjoint union of a pair of pants and two once-punctured tori, and hence has norm $3$.

\begin{figure}[htbp]
\centering
\includegraphics[width=.88\textwidth]{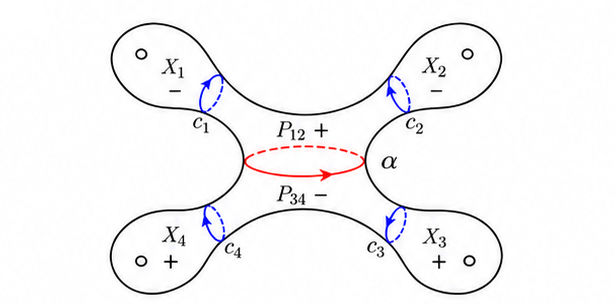}
\caption{The genus-four configuration. The signs indicate the components of $R_+(\gamma)$ and $R_-(\gamma)$; the colored curves indicate the sutures.}
\label{fig:sutures}
\end{figure}

\begin{lemma}\label{lem:taut}
The sutured manifold $(M,\gamma)$ is taut.
\end{lemma}

\begin{proof}
Since $M$ is simple, it is irreducible and boundary-irreducible; in particular, the torus components in $T(\gamma)$ are incompressible. All boundary curves of the components of $R(\gamma)$ are essential in $F$, so these components are incompressible in $F$. Since $F$ is incompressible in $M$, they are also incompressible in $M$. It remains only to prove that $R_+(\gamma)$ and $R_-(\gamma)$ are Thurston-norm minimizing.

Let $S\subset M$ be an oriented properly embedded surface with the same oriented boundary as $R_+(\gamma)$ and representing the same relative homology class as $R_+(\gamma)$. Its boundary consists of five pairwise nonparallel essential curves parallel to $\alpha,c_1,c_2,c_3,c_4$. Write $S=\bigsqcup_jS_j$, with each $S_j$ connected.

\medskip
\noindent\emph{Claim 1.} Every component $S_j$ with nonempty boundary has positive norm.

\smallskip
\noindent\emph{Proof of Claim 1.}
A disk component would give a compressing disk for $F$, contrary to boundary-irreducibility. An annular component is incompressible, for otherwise a compression would yield a disk in $M$ bounded by one of its boundary components, contradicting the incompressibility of $F$. The two boundary curves of such an annulus are among the five pairwise nonparallel components of $\partial S$, so it cannot be boundary-parallel. It would therefore be an essential annulus in $M$, contradicting anannularity. Hence $\chi(S_j)<0$, and $x(S_j)>0$. \qed

Suppose, for contradiction, that $x(S)\le2$. By Claim~1 and the fact that $\partial S\ne\varnothing$, $S$ has either one or two components with nonempty boundary. We distinguish these two cases.

\smallskip
\noindent\emph{Case 1.} Exactly one component of $S$ has nonempty boundary. It has all five boundary components of $S$. If its genus is $h$, then its norm is $2h+3\ge3$, contradicting $x(S)\le2$.

\smallskip
\noindent\emph{Case 2.} Exactly two components of $S$ have nonempty boundary. Since $x(S)\le2$, each has norm $1$. A connected orientable surface of norm $1$ has an odd number of boundary components. Hence the total number of boundary components of these two surfaces is even, contradicting $|\partial S|=5$.

Therefore $x(S)\ge3=x(R_+(\gamma))$, so $R_+(\gamma)$ is Thurston-norm minimizing. The same argument applies to $R_-(\gamma)$. Thus $(M,\gamma)$ is taut.
\end{proof}

\subsection{Proof of Theorem~\ref{thm:main}}

\begin{proof}
Let $H_\beta$ be the $2$--handle attached along $\beta$. Choose a reducing sphere $\widehat Q\subset M(\beta)$ minimizing the number of meridian disks in $\widehat Q\cap H_\beta$, and put $Q=\widehat Q\cap M$. The surface $Q$ is connected and planar, with all boundary components parallel to $\beta$. Writing $q=|\partial Q|$, we have $\chi(Q)=2-q$; irreducibility of $M$ implies $q\ge1$. The proof of Lemma~2.1 in \cite{QZ} shows that the minimality of $|\widehat Q\cap H_\beta|$, together with the boundary-irreducibility of $M(\beta)$, implies that $Q$ is incompressible and boundary-incompressible.

To verify the hypotheses of Theorem~\ref{thm:Taylor}, isotope the components of $\partial Q$ to parallel copies of $\beta$ in $P$, disjoint from the annuli about $c_1,\ldots,c_4$ and meeting the annulus about $\alpha$ in two spanning arcs. Thus condition~(P1) holds. Since $q\ge1$, $Q$ is not a sphere; if $Q$ is a disk, then its boundary meets $\alpha\subset\gamma$. Hence condition~(P2) holds, and $Q$ is a parameterizing surface, with
\[
Q\cap(\gamma-\alpha)=\varnothing,\qquad
|Q\cap\gamma|=|Q\cap\alpha|=2q.
\]
By construction, the portions of $\partial Q$ lying in the two components of $R(\gamma)$ adjacent to $\alpha$ are essential arcs. These observations, together with Lemma~\ref{lem:taut}, verify all the hypotheses of Theorem~\ref{thm:Taylor}, which we apply with $N=M$ and $b=\alpha$.

The first three alternatives are ruled out in turn. Since $Q$ is incompressible, alternative~(1) would yield an $\alpha$--boundary-compressing disk $D$ with $\partial D=a\cup d$, where $a\subset Q$ and $d\subset\alpha$. The arc $a$ cannot be essential in $Q$, for otherwise $D$ would be an ordinary boundary-compressing disk for $Q$. Thus $a$ cobounds a disk $E\subset Q$ with an arc $e\subset\partial Q$. After a small perturbation, $D\cup_aE$ is a properly embedded disk in $M$ with boundary $d\cup e$. Boundary-irreducibility of $M$ implies that $d\cup e$ is inessential in $F$. Choose the disk $B\subset F$ bounded by $d\cup e$ innermost with respect to $\alpha\cup\partial Q$. Then $B$ is a bigon between $\alpha$ and the corresponding component of $\partial Q$, contradicting minimal position. Thus alternative~(1) is impossible.

If alternative~(2) holds, let $S\subset M(\alpha)$ be the corresponding summing sphere, so that $|S\cap\kappa_\alpha|=2$, and let $H$ denote the attached $2$--handle. Isotope $S$ so that $S\cap H$ consists of two meridian disks. Then $A=S\cap M$ is a properly embedded annulus whose boundary components are parallel to $\alpha$. The incompressibility of $F$ and essentiality of $\alpha$ imply that $A$ is incompressible; anannularity of $M$ therefore forces $A$ to be boundary-parallel to an annulus $A_0\subset F$. The product region between $A$ and $A_0$ is a solid torus in which $\alpha$ is a longitude. Attaching the portion of $H$ between the two disks of $S\cap H$ gives a $3$--ball bounded by $S$, in which the corresponding subarc of $\kappa_\alpha$ is boundary-parallel. This contradicts Remark~\ref{rmk:Taylor-summand}, so alternative~(2) is impossible.

Alternative~(3) is also impossible: $M(\alpha)$ is reducible, whereas the underlying manifold of a taut sutured manifold is irreducible. Only alternative~(4) remains. Substituting $\chi(Q)=2-q$ and the intersection counts above gives
\[
4q-4=-2\chi(Q)+|Q\cap\gamma|
   \ge2|Q\cap\alpha|=4q,
\]
a contradiction. This proves Theorem~\ref{thm:main}.
\end{proof}

\section*{Acknowledgements}

I am grateful to Ruifeng Qiu for suggesting, more than a decade ago, that I use Taylor's methods to study handle additions.

\section*{Declaration on the use of AI}

During the preparation of this work, the author used artificial intelligence (AI) tools solely for English-language polishing. After using these tools, the author reviewed and edited the content as needed and takes full responsibility for the content of the manuscript.

\end{document}